\documentclass{article}
\usepackage[utf8]{inputenc}
\usepackage{amsmath,amsthm}
\usepackage{amssymb}
\usepackage{xcolor}
\usepackage[capitalise]{cleveref}
\newcommand{\R}{\mathbb{R}}
\newcommand{\C}{\mathbb{C}}
\newcommand{\N}{\mathbb{N}}
\newcommand{\E}{\mathbb{E}}
\newcommand{\U}{\mathcal{U}}
\renewcommand{\H}{\mathbb{H}}
\renewcommand{\O}{\mathbb{O}}
\newcommand{\B}{\mathcal{B}}
\newcommand{\Bo}{\mathcal{B}_\O}
\newcommand{\Boa}{\mathcal{B}_{\O1}}
\newcommand{\Bob}{\mathcal{B}_{\O2}}
\newcommand{\Bh}{\mathcal{B}_\H}

\newcommand{\norm}[1]{\left\lVert#1\right\rVert}

\theoremstyle{theorem}
\newtheorem{theorem}{Theorem}

\newtheorem{thm}{Theorem}
\theoremstyle{prop}
\newtheorem{prop}[thm]{Proposition}

\theoremstyle{requirement}
 \newtheorem{requirement}{Axiom}

\theoremstyle{property}
\newtheorem{property}{Property}

\newtheorem*{main_question*}{Main Question}

\newcommand{\ignore}[1]{}

\usepackage[margin=1.25in]{geometry}
\crefname{prop}{Proposition}{Propositions}

\makeatletter
\let\xx@thm\@thm
\AtBeginDocument{\let\@thm\xx@thm}
\makeatother

\makeatletter
\let\xx@prop\@prop
\AtBeginDocument{\let\@prop\xx@prop}
\makeatother

\makeatletter
\let\xx@cor\@cor
\AtBeginDocument{\let\@cor\xx@cor}
\makeatother
	
\title{ Elementary Derivations of the Real Composition Algebras\\ 
	\vspace{5pt}
  \large{Adapted from Gadi Moran's last paper}}

\author{
     Tomer Moran\thanks{McGill University} 
	\and Shay Moran\thanks{Mathematics Department, Technion Israel, and Google Research}
    \and Shlomo Moran \thanks{CS Department, Technion Israel}}
\date{\today}

\begin{document}

\maketitle

\abstract{

{``Real Normed Algebras Revisited,'' the last paper of the late Gadi Moran, attempts to reconstruct the discovery of the complex numbers, 
	the quaternions and the octonions, as well as proofs of their properties, using only what was known to 19th century mathematicians. 
	In his research, Gadi had discovered simple and elegant proofs of the above-mentioned classical results 
	using only basic properties of the geometry of Euclidean spaces and tools from high school geometry. 
	His reconstructions underline an interesting connection between Euclidean geometry and these algebras,  
	and avoid the advanced machinery used in previous derivations of these results. 
	The goal of this article is to present Gadi's derivations in a way that is accessible to a wide audience of readers.}

\section{Introduction}
Inspired by how complex numbers could be represented as points {in the} plane, 
	Sir William Rowan Hamilton attempted to create an algebra\footnote{An algebraic structure that supports addition and multiplication, to be defined soon.} 
	in 3-dimensional space. Unaware of the fact that his requirements were impossible to meet, 
	it took him some time and effort before he eventually realized, 
	in 1843, that such an algebra required a fourth dimension. 
	Hamilton then discovered the quaternions, a noncommutative 4-dimensional algebra, 
	and he devoted the remainder of his life to studying and teaching them \cite{gadis, vdw76}. 
	Inspired by Hamilton's discovery, Graves discovered shortly thereafter the octonions \cite{graves1845}, 
	and at about the same time they were discovered also by Cayley \cite{cayley1845}. 
	Subsequently it was proved that the complex numbers, 
	the quaternions, and the octonions are the only possible extensions of the real numbers 
	to  {algebras} that satisfy certain natural properties. 
	Excellent expositions of these discoveries and their applications can be found in \cite{Baez02} and \cite{conway2003quaternions}. %

In an attempt to reconstruct Hamilton’s discovery of the quaternions 
	using only what was then known to Hamilton and to other mathematicians of the 19th century, 
	the late Gadi Moran developed an alternative derivation of {these structures.  
	Specifically, Gadi's derivations assume the structure of a {\it Hurwitz algebra} with respect to the standard Euclidean norm, i.e., {\it Euclidean} Hurwitz algebras.}
	Interestingly, his proofs rely only on elementary properties of the Euclidean {distance}, 
	which makes them accessible and easy to appreciate by students who first encounter the topic. 
	As Gadi mentions~\cite{gadis}, it turns out that his derivation significantly diverges from Hamilton’s, as described in \cite{vdw76}. 
	Moreover, Gadi used his technique to derive a complete characterization of Euclidean Hurwitz algebras, 
	thus proving  {Hurwitz's celebrated theorem} \cite{Hurwitz1922} using only tools of high-school geometry.

  {Unfortunately, Gadi had tragically and unexpectedly passed away on January 1, 2016 in a train accident, before he was able to publish his work. 
	Since Gadi was always enthusiastic about exposing the beauty of mathematics to a wide audience, we, as his close relatives (grandson, nephew  and brother), have taken upon ourselves to
	write this article, with the aim of making Gadi's proofs accessible to people who master high school mathematics.} 
	The definitions of the relevant algebraic concepts---e.g., {\em field, linear space, norm, algebra}---are usually studied in a first-year course of college mathematics.

\paragraph{Organization.} 
We first formally state the requirements for Euclidean Hurwitz algebras in \cref{sect:premises}. Then, we present Gadi’s derivation of complex numbers in \cref{sect:complex}; his derivation of quaternions in \cref{sect:quaternions}; his derivation of octonions  in \cref{sect:octonions};  and finally his proof of Hurwitz’s theorem is completed in \cref{sect:sedenions}.
In \Cref{sec:history} we discuss the relationship between the classical theorems of Hurwitz and  Frobenius  
and Gadi's development of the complex numbers, the quaternions, and the octonions.

\section{Premises} \label{sect:premises}
\paragraph{Geometry of numbers.}

Our starting point is $\R$, the {\em field} of real numbers, which (like any field) is equipped with addition and multiplication. Geometrically, $\R$ is naturally identified with the one-dimensional Euclidean space\footnote{By $n$-dimensional Euclidean space, we mean a vector space over $\R^n$ with the standard Euclidean distance.} (i.e., the real line), denoted by $\E^1$. An extension of paramount importance of $\R$ is $\mathbb{C}$, the field of complex numbers. Analogously, $\mathbb{C}$ is identified with the two-dimensional Euclidean space $\E^2$ (i.e., the real plane). The latter identification follows from associating the complex number~$x+y\cdot i\in\mathbb{C}$ with the point $(x,y)\in\E^2$. This suggests the following question, which motivates the study of Euclidean Hurwitz algebras:
\begin{quote}
For which values of $n \in \mathbb{N}$ can the $n$-dimensional Euclidean space $\E^n$ be equipped with multiplication?
\end{quote}
The fields $\R$ and $\mathbb{C}$ witness that it can be done for $n=1,2$. To address this question more systematically, one needs to specify which properties this multiplication should satisfy. We address this in the next section.

\subsection{Formal requirements}

Our goal is to extend the vector space $\E^n$ to an algebraic structure that also supports a multiplication. In this section we specify three basic axioms we require of the desired multiplication. The first two axioms are {\it algebraic}, and they can be imposed on vector spaces over arbitrary fields  (in modern algebra they capture a structure called {\it algebra}.) 
The third axiom is {\it geometric} in the sense that it hinges on the Euclidean norm, which is special to real vector spaces. All three axioms are trivially satisfied by the complex plane and the real line.

It is natural to require the {\it distributive law}---the basic property that ties multiplication and addition:



\begin{requirement}[Left and Right Distributive Laws]\label{axiom:dist}
For all $x,y,z\in\E^n$,
\begin{align*}
x\cdot (y+z)&=x\cdot y+x\cdot z;\\
 (y+z)\cdot x &=y\cdot x+z\cdot x.
 \end{align*}
\end{requirement}


The following axiom, which is a basic tool in our geometric construction,  identifies a {\em scalar} $a\in\R$   with the vector $(a,0,\ldots,0)\in \E^n$ in the sense that   multiplying~$x=(x_1,\ldots,x_n)\in E^n$ with $(a,0,\ldots,0)$ agrees with the standard scalar multiplication of $x$ by $a$. We note that \cref{axiom:homo} is equivalent to assuming that multiplication is homogeneous and that there exists a {\em unit vector} $e$ such that $e\cdot x=x\cdot e=x$ for all $x\in E^n$. 
A didactic derivation of this axiom and motivation for it can be found in Gadi's work \cite{gadis}. 
\begin{requirement}[Scalars as Vectors]\label{axiom:homo}
	For all $x,y\in\E^n$ and $a\in\mathbb{R}$:
	\[a(xy) = (ax)y = x(ay),\]
	where $a\in\mathbb{R}$ is identified with\footnote{The choice of $(a,0,\ldots,0)$ is done for the sake of concreteness.
	Alternatively, one can identify the scalars with any 1-dimensional subspace of $\E^n$, i.e., a line through the origin.}
	$(a,0,\ldots,0)\in\E^n$.
\end{requirement}


\vspace{1mm}

Last but not least, we require that the sought multiplication is compatible with the geometry of $\E^n$ in the following sense.
Recall that $\E^n$ is equipped with the Euclidean norm, which represents length, and is defined by
\[\norm{x}=\norm{(x_1,\ldots,x_n)}=\sqrt{x_1^2+\cdots+x_n^2}.\]
The desired multiplication is required to satisfy the following {\em multiplicity of norm} rule:

\begin{requirement}[Multiplicity of Norm (MoN)]\label{axiom:mon}
For all $x,y\in\E^n$, 
\[\norm{xy}=\norm{x}\norm{y}.\]
That is, the norm of a product is the product of the norms. 
\end{requirement}


Summarizing the above, our task reduces to the following question: 
\begin{center}
\noindent\fbox{
\parbox{.875\columnwidth}{
\begin{main_question*}
For what values of $n>1$ it is possible to define multiplication over~$\E^n$ that satisfies the following three properties?
\begin{itemize}
\item[(i)] left and right distributivity (\Cref{axiom:dist}), 
\item[(ii)] scalars as vectors (\Cref{axiom:homo}), and 
\item[(iii)] Multiplicity of Norm (MoN) rule (\Cref{axiom:mon}).
\end{itemize}
\end{main_question*}
}}
\end{center}
The complete answer to this question, which was explored by Hamilton and his contemporaries and reproduced here, shows that $\C$ is the only Euclidean Hurwitz algebra over $\E^2$, and that the only other possible Euclidean Hurwitz algebras are the quaternions over $\E^4$, in which multiplication is not commutative, and the octonions over $\E^8$, in which multiplication is neither commutative or associative.

\subsection{Useful properties}
The proofs presented use elementary geometric arguments. Specifically, they use solely \Cref{property:p1} and \Cref{property:p2} below,
both of which are implied by the next versions of the triangle inequality and the Pythagorean theorem.\footnote{Note that the Pythagorean theorem is implied by the Euclidean norm.}
(i) The triangle inequality: $\norm{x}+ \norm{y} \geq \norm{x+y}$ for all $x,y\in \E^n$,
with equality if and only if  $x,y$ lie on the same ray  extended from the origin,
and (ii) the Pythagorean theorem: $\norm{x}^2+\norm{y}^2 = \norm{x+y}^2$ if and only if $x \perp y$, i.e., $x$ and $y$ are orthogonal, or~$x$ is perpendicular to $y$.  

\begin{property}[Equality Statement]\label{property:p1}
Let $x,y \in \E^n$; then
\[\Bigl(\norm{x+y}=\norm{x}+\norm{y}\quad \text{ and }\quad\norm{x}=\norm{y}\Bigr)\iff \bigl(x=y\bigr).\]
\end{property}
Note that \Cref{property:p1} is a corollary of the version of the triangle inequality stated above. 
 {Indeed, the direction~``$\Longleftarrow$'' is trivial, and the direction ``$\implies$'' follows since by the triangle inequality $\norm{x}+\norm{y}=\norm{x+y}$ implies that
$y=\alpha\cdot x$ for some $\alpha \geq 0$, and $\norm{x}=\norm{y}$ implies that $\alpha=1$, hence $x=y$.
}

\begin{property}[Orthogonality Statement]\label{property:p2}
If $\norm{x}=\norm{y}=1$; then 
\[x\perp y \Longleftrightarrow \norm{x+y}=\sqrt{2}.\]
\end{property}
Note that \Cref{property:p2} is a corollary of the Pythagorean theorem.

\paragraph{Enough to define a multiplication over a basis.}
Since bilinearity (i.e., the left and right distributive laws and homogeneity) is always required, a multiplication over $\E^n$ spaces is determined entirely by its definition over any basis of $\E^n$. Indeed, given a basis $e_1, e_2, \dots, e_n$, any $x\in\E^n$ is a linear combination of the form
 \[x=\alpha_1 e_1 + \alpha_2 e_2 + \dots + \alpha_n e_n, 
 ~~\mbox{	where } \alpha_i \in \R ~\mbox{ are real coefficients.}\]

For the rest of the article, elements of $\E^n$ are represented by the  {\em standard} orthonormal basis $\{e_1,\ldots,e_n\}$, where $e_i$ has 1 in the $i$th entry and 0 otherwise. We remind the reader that the element $e_1=(1,0,\ldots,0)$ has a special role: it is the unit element in the assumed Euclidean Hurwitz algebra. (See \Cref{axiom:homo}.)


With these postulates in hand we are now ready to extend the real numbers to algebras of higher dimensions in a systematic way.  {We will derive the necessary conditions for these extensions. The sufficiency of these conditions (specifically checking that the MoN rule (\Cref{axiom:mon}) is satisfied by each given extension) can be verified by brute-force calculations.}

\section{The Complex Numbers} \label{sect:complex}
To demonstrate our approach we show how the complex numbers are naturally and easily derived from \Cref{property:p2} and MoN. 
 We start with a general useful result for Euclidean Hurwitz algebras over  $\E^n,n\in\N$.

Let $\U_n = \R^\perp$ denote the orthogonal complement of $\R$ in $\E^n$. That is, $\U_n$ is the set of numbers in $\E^n$ that are orthogonal to all real numbers, i.e., the subspace of $\E^n$ spanned by $\{e_2,\ldots,e_n\}$.

\begin{prop}\label{prop:imaginary_squares}
	In any Euclidean Hurwitz algebra over $\E^n$, if $u \in \U_n$ and $\norm{u} = 1$, then  $u^2=-1$.
\end{prop}

\begin{proof} Consider the product $(u+1)(u-1)$. By MoN   and \Cref{property:p2}, we can evaluate its norm:
\begin{align*}
\norm{(u+1)(u-1)} &= \norm{u+1} \norm{u-1} \tag{MoN  }\\
			   &= \sqrt{2} \sqrt{2} = 2.        \tag{\Cref{property:p2}}
\end{align*}
The product can also be simplified by distributing: $(u+1)(u-1)=u^2-1$. Therefore we have
\[\norm{u^2-1}=2=1+1=\norm{u^2} + \norm{-1},\]
 where the last step uses MoN to evaluate $\norm{u^2} =\norm{u}^2 = 1$. 
We can rewrite the above as
\[\norm{u^2 + (-1)} = \norm{u^2} + \norm{-1}.\]
Given that $\norm{u^2}=1$, we can apply \Cref{property:p1} and deduce that $u^2 = -1$.  
\end{proof}

We next explain how \cref{prop:imaginary_squares} easily implies the existence and uniqueness of a Euclidean Hurwitz algebra over the Euclidean plane $\E^2$,
namely the well-known algebra of complex numbers $\C$. Let $i$ be the (number associated with the) unit vector $e_2=(0,1)\in\E^2$. Then $i\in\U_2$ and hence,  by Proposition \ref{prop:imaginary_squares}, $i^2=-1$ in any Euclidean Hurwitz algebra over $\E^2$. In addition $\{1,i\}=\{e_1,e_2\}$ is a basis of $\E^2$. Hence any  number $Z$ in such an algebra can be represented as  $Z=\alpha + \beta i$ for some $\alpha,\beta\in\R$. Thus, the only multiplication over $\E^2$ that satisfies the formal requirements in \cref{sect:premises} is the well-known multiplication over $\C$ defined by
  $$(\alpha +\beta i)(\gamma+\delta i) = (\alpha\gamma-\beta\delta) + (\alpha\delta + \beta\gamma)i.$$ 
  This proves that $\C$ is the only Euclidean Hurwitz algebra over $\E^2$, as claimed.

\section{The Quaternions} \label{sect:quaternions}

We now aim to extend the complex numbers to Euclidean Hurwitz algebras over $\E^n$ for $n>2$, or to prove that no such extension exists for a given $n$. We do so by defining multiplication over $\E^n$, or else show that such a definition is impossible.

A Euclidean Hurwitz algebra over $\E^n$ for $n>2$ must contain numbers $u,v \in \U_n$ such that $\{1,u,v\}$ is an orthonormal set. 
By Proposition \ref{prop:imaginary_squares} we get that $u^2=v^2=-1$. We next realize that $uv=-vu$, which implies that commutativity does not hold for Euclidean Hurwitz algebras over $\E^n$ if $n>2$.

\begin{prop}\label{prop:anticommutativity}
If $u,v \in \U_n$ are orthogonal and $\norm{u}=\norm{v}=1$, then $uv=-vu$.
\end{prop}

\begin{proof}
Since $u,v\in \U_n$ and $\norm{u}=\norm{v}=1$, \cref{prop:imaginary_squares} implies that $u^2=v^2=-1$, hence $(u-v)(u+v)=uv-vu$. Also, since $u$ and $v$ are orthonormal, \Cref{property:p2} implies that $\norm{u+v}=\norm{u-v}=\sqrt{2}$. Thus we get
\begin{equation}\label{eq:uv1}
\norm{uv-vu} =\norm{(u-v)(u+v)}=\norm{(u-v)}\norm{(u+v)} =2= \norm{uv} + \norm{-vu},
\end{equation}
where the rightmost equality uses MoN  to evaluate $\norm{uv}=\norm{-vu}=\norm{u} \norm{v}=1$. \Cref{eq:uv1} implies that
$\norm{uv + (-vu)} = \norm{uv} + \norm{-vu}$,  and
 we apply \Cref{property:p1} to deduce that $uv=-vu$. 
\end{proof}

 Note that \cref{prop:anticommutativity} holds even if the norm of $u, v$ is not 1 (so long as it is nonzero). 
 Next we show that if $\{1,u,v\}$ is an orthonormal set in $\E^n$, then $uv$ must be a unit vector orthogonal to all three vectors. Since there is no such vector when $n=3$, we conclude that there is no Euclidean Hurwitz algebra over $\E^3$.

\begin{prop}\label{prop:fourth_dimension}
If $u,v \in \U_n$ are orthonormal, then $uv \in \U_n$, and $uv$ is orthogonal to $1,u,$ and $v$.
\end{prop}

\begin{proof} By \cref{prop:imaginary_squares} we get that $uv-1 = u(v+u)$. Thus the  norm of $uv-1$ can be evaluated using MoN  and \Cref{property:p2}:
\begin{equation}
\norm{uv-1}=\norm{u(v+u)}=\norm{u} \norm{v+u} = 1\sqrt{2} = \sqrt{2} .
\end{equation}
Since $\norm{uv+(-1)} = \sqrt{2}$,   \Cref{property:p2} implies that $uv$ is orthogonal to $-1$. This means that $uv \in \U_n$, and hence by Proposition \ref{prop:imaginary_squares} that $(uv)^2 = -1$. 

Similarly, applying MoN and \Cref{property:p2} to the expression $uv+u$ we get:
\[\norm{uv+u}=\norm{u(v+1)}=\norm{u} \norm{v+1} = 1\sqrt{2} = \sqrt{2} ,\]
which implies by \Cref{property:p2} that $uv$ is orthogonal to $u$. The same reasoning can be applied with $uv+v$ to finally show that $uv$ is orthogonal to $v$. 
\end{proof}

We have shown that if there exists a Euclidean Hurwitz algebra over $\E^n$, then products in $\U_n$ are anti-commutative, and that no real Euclidean Hurwitz algebra exists for $n=3$. Indeed, \cref{prop:fourth_dimension} implies that introducing a second imaginary dimension implicitly adds a third one, as shown by the fact that the number $uv$ is perpendicular to $1, u$, and $v$. We claim that these four vectors $\Bh=\{1,u,v,uv\}$ form an orthonormal basis for the 4-dimensional Euclidean space $\E^4$---to complete the proof that this indeed forms a Euclidean Hurwitz algebra, it remains to show that $\Bh$ is closed under multiplication; that is, the product of elements in $\E^4$ is also in $\E^4$. To prove the latter we first derive the following useful identity.

\begin{prop}\label{prop:identity}
    If $x,y \in \U_n$ are orthonormal, then $(xy)x=x(yx)=y$.
\end{prop}

\begin{proof}
	 Consider the product $\bigl((1+x)(1+y)\bigr)(1+x)$:
\begin{eqnarray}\label{eq:3}
    \bigl((1+x)(1+y)\bigr)(1+x) &=& (1+x+y+xy)(1+x)\nonumber \\
    &=& 1+x+y+xy+x+x^2+yx+(xy)x \nonumber\\
    &=& (x+y) + (1+xy)x.
\end{eqnarray}
Here the last step uses $xy=-yx$ and $x^2=-1$. By MoN  and equation (\ref{eq:3}) we obtain
\[2\sqrt{2}=\sqrt{2}\sqrt{2}\sqrt{2}=\norm{\bigl((1+x)(1+y)\bigr)(1+x)} = \norm{(x+y) + (1+xy)x}.\]
Next observe that $\norm{x+y}=
\sqrt{2}$ by \Cref{property:p2}, and $\norm{(1+xy)x}=\sqrt{2}$ by  \Cref{property:p2} and MoN. Thus we get
\begin{equation}\label{eq:2sqrt2}
    \norm{x+y} + \norm{(1+xy)x}  =2\sqrt{2}= \norm{(x+y) + (1+xy)x}.
\end{equation}
By \Cref{property:p1}, equation (\ref{eq:2sqrt2}) implies that the two terms are equal:
\[x+y = (1+xy)x = x+(xy)x,\]
which directly implies that $y = (xy)x$. Finally, two applications of \cref{prop:anticommutativity} yield $x(yx)=-(yx)x=(xy)x$.
\end{proof}

From \cref{prop:identity} we can easily show that $\E^4$ is closed under multiplication. Indeed, \Cref{prop:identity} (combined with \Cref{prop:anticommutativity} and \Cref{prop:imaginary_squares}) imply that the product of any two elements in the basis $\Bh = \{ 1, u, v, uv\}$ is in $\{\pm 1,\pm u, \pm v, \pm uv\}\subseteq \E^4$; by bilinearity and homogeneity the closure extends to the entire space. This implies the existence of a 4-dimensional Euclidean Hurwitz algebra, which we denote as $\H$. To show that this corresponds precisely to the quaternions, we first show that multiplication in $\H$ is associative.

\begin{prop}\label{prop:associativity}
    Products in $\H$ are associative. That is, for any $x,y,z \in \H$, $x(yz) = (xy)z$.
\end{prop}

\begin{proof} 
As discussed at the end of \cref{sect:premises}, due to the bilinearity it suffices to prove associativity for products of elements in an orthonormal basis of $\H$, such as $\Bh= \{ 1,u,v,uv\}$.
The proof is trivial when either $x,y,$ or $z$ equal 1, by the properties of scalar multiplication (\Cref{axiom:homo}). 
Thus, we consider only triples in $\{u,v,uv\}$. For triples where a factor is repeated:
\begin{itemize}
    \item $x(xx) = x(-1) = (-1)x=(xx)x$ \hfill (by \cref{prop:imaginary_squares})
    \item $(xy)x = (-yx)x = -(yx)x = x(yx)$ \hfill (by \cref{prop:anticommutativity,prop:identity})
    \item $x(xy) = -(xy)x = -y =(xx)y  $ \hfill (by \cref{prop:anticommutativity,prop:identity})
\end{itemize}
   Of the $3!=6$ possible remaining products, we only prove 3; the other proofs follow essentially the same structure. All of the following derivations use \cref{prop:imaginary_squares,prop:anticommutativity,prop:fourth_dimension,prop:identity}. 
To avoid ambiguity, we use $[uv]$ to denote the element $uv\in\Bh$. 
\begin{itemize} 
    \item $u(v[uv]) = u(u)=-1=[uv][uv]=(u v)[uv]$
    \item $[uv](vu) = -[uv](uv) =  1 = -(uu) = -([vu]v)u=([uv]v)u$
    \item $ v([uv]u) = vv= -1 = uu = (v[uv])u$
\end{itemize}
 \end{proof}
We have thus completed the discovery of the quaternions $\H$, and have proved that they form the only Euclidean Hurwitz algebra over the 4-dimensional Euclidean space. By relabeling $i=u$, $j=v$, $k=uv$, we recover Hamilton's familiar equations of quaternions:
\[i^2=j^2=k^2=ijk=-1.\] 

\section{The Octonions} \label{sect:octonions}

Any Euclidean Hurwitz algebra over $\E^n$ for $n>4$ contains the quaternions, and in addition it also contains a unit vector $w$ that is perpendicular to all vectors in $\Bh$. The next proposition shows that such an algebra must contain eight orthonormal vectors, i.e., its dimension is at least 8.
\begin{prop}\label{prop:eighth_dimension}
	Let $w$ be a unit vector that is orthogonal to all vectors in $\Bh$. Then the set $\Bo=\{1, u, v, uv, \\ w, uw, vw, (uv)w\}$ is orthonormal.
\end{prop}
\begin{proof}
	We use the following observation, implied by MoN, \Cref{property:p2}, and the distributive laws:
	\begin{equation}\label{eq:3perp}
 \mbox{For any three unit vectors}~x,y,z:~~x\perp y\Leftrightarrow xz\perp yz \Leftrightarrow zx\perp zy .\footnote{One can show that equation (\ref{eq:3perp}) also holds when $x,y,z$ are not necessarily unit vectors.}
	\end{equation} 
Let $\Boa = \Bh= \{1, u, v, uv\}$ and $\Bob = \{w, uw, vw, (uv)w\}$. We already know that $\Boa$ is orthonormal, which immediately implies by equation (\ref{eq:3perp}) that $\Bob$ is orthonormal as well.

It remains to show that each vector in $\Boa$ is perpendicular to all vectors in $\Bob$, namely that for any pair $x,y\in\Boa$, $x$ is orthogonal to $yw$. So let $x,y\in\Boa$ be given. It is easily verified that for some $z\in\Boa$ we have that $x=yz$ or $x=-yz$, so assume without loss of generality that $x=yz$ (the other case is similar). Since $z$ is orthogonal to $w$, by equation (\ref{eq:3perp}) we get that $x=yz$ is orthogonal to $yw$.
\end{proof}

It turns out that  $\Bo=\{1, u, v, uv, w, uw, vw, (uv)w\}$, the orthonormal set defined in \cref{prop:eighth_dimension} above, forms a basis for a Euclidean Hurwitz algebra over the 8-dimensional Euclidean space $\E^8$. This algebra is called the octonions and denoted by $\O$.  To see that $\O$ is indeed a Euclidean Hurwitz algebra it suffices to show that it is closed under multiplication. The product of two elements in $\Boa$ is inherited from the quaternions. For the other products we use the following useful identity:
\begin{prop}\label{prop:identity2}
	Let $x,y,z \in \U_n$ be orthonormal vectors such that $xy$ is orthogonal to $z$. Then $(xy)(yz)=xz$.
	\end{prop}
\begin{proof}
	The following equalities hold for any set of orthonormal vectors $\{x,y,z\}$: 
	\begin{eqnarray}\label{eq:xzw}
	 (xy-z)(x+yz)&=&(xy)x+(xy)(yz)-zx-z(yz)  \\
	 &=&y+(xy)(yz)+xz-y  \nonumber \\
	&=&(xy)(yz)+xz, \nonumber
	\end{eqnarray}
	where the first and last terms were simplified using  \cref{prop:anticommutativity,prop:identity}. 
	Since $xy$ is orthogonal to $z$ we also have by \cref{prop:identity} that $x$ is orthogonal to $yz$. Thus, by MoN  and \Cref{property:p2},   
	\[\norm{(xy-z)(x+yz)}=\sqrt{2}\sqrt{2}=2.\]
	Thus we conclude from equation (\ref{eq:xzw}) that
	\[2=\norm{(xy)(yz)+xz}=\norm{(xy)(yz)}+\norm{xz} ,\]
	which by \Cref{property:p1} implies that $(xy)(yz)=xz$.  
\end{proof}
 
We now show that the product of two distinct elements in $\Bob$, or its negation, is in $\Boa$. Let these elements be $xw,yw$, for some distinct $x,y\in\Boa$. If $x=1$ or $y=1$, then this follows by \cref{prop:identity}.  Otherwise, $x,w$, and $y$ satisfy the hypotheses of \cref{prop:identity2}, hence $(xw)(yw)=-(xw)(wy)=-xy$ and we are done.

It remains to show that products between elements of $\Boa$ and $\Bob$, or their negations, are in $\Bo$. This is accomplished by the following proposition, which as a by-product demonstrates that certain products in $\O$ are {\em anti-associative}.
\begin{prop}\label{prop:antiassociativity}
For any distinct $x,y \in \Boa$, we have $x(yw)=-(xy)w$.
\end{prop} 

\begin{proof}
	 First observe that $x,y$ and $w$ satisfy the hypotheses of \cref{prop:identity2}. Hence,
\begin{align*}
    (x-xy)(yw+w)&=x(yw)+xw-(xy)(yw)-(xy)w \\
    &=x(yw)+xw-xw-(xy)w \\
    &=x(yw)-(xy)w .
\end{align*}

 By MoN   and \Cref{property:p2}, we evaluate the norm:
\[\norm{(x-xy)(yw+w)}=\sqrt{2}\sqrt{2}=2.\]
Comparing with the reduced expression, we have
\[\norm{x(yw)-(xy)w}=\norm{x(yw)}+\norm{-(xy)w}.\]
By \Cref{property:p1}, we find that $x(yw)=-(xy)w$. 
\end{proof}

\cref{prop:identity2,prop:antiassociativity} above complete the definition of products over $\Bo$, the eight basis vectors of the octonions. Our derivation demonstrates that 
the octonions are the (only) Euclidean Hurwitz algebra over the 8-dimensional Euclidean space. During this process, we have also shown that no Euclidean Hurwitz algebra can exist in 5, 6, or 7 dimensions. We also note that unlike the reals, the complex numbers, and the quaternions, multiplication over the octonions is nonassociative. 

\section{Extending the Octonions} \label{sect:sedenions}

We next push our method further to check if there exists a Euclidean Hurwitz algebra over $\E^n$ for $n>8$. Again, such an algebra must contains the octonions, and in addition a unit vector $s \in \U_n$ that is orthogonal to all the vectors in $\Bo$. By repeating the arguments in \cref{prop:eighth_dimension}, we obtain an orthonormal set $\B'$ of 16 vectors, that we (wrongly) assume is a basis for a Euclidean Hurwitz algebra over $\E^{16}$:
\[\B'=\Big\{ 1,u,v,uv, \hspace{12pt} w,uw,vw,(uv)w, \hspace{12pt} s,us,vs,(uv)s, \hspace{12pt} ws,(uw)s, (vw)s, ((uv)w)s \Big\}.\]

\begin{prop}\label{prop:no16}
It is impossible to define a bilinear and homogeneous multiplication that satisfies MoN, \Cref{property:p1}, and \Cref{property:p2} over the $n$-dimensional Euclidean space, for $n>8$. 

\end{prop} 

\begin{proof}
Assume that it is possible to define such a product for some $n>8$. Then $\E^n$ with this product must contain an orthonormal set  $\B'$  as given above. Then we get
 {
\begin{align*}
     (uv+ws)(sv+wu)&=(uv)(sv)+(uv)(wu)+(ws)(sv)+(ws)(wu) \\
    &= -us + vw + wv - su \tag{by \cref{prop:anticommutativity,prop:identity2}} \\
    &= -us - wv + wv + us = 0    \tag{by \cref{prop:anticommutativity}}.
\end{align*}
}%
However, since $uv\perp ws$ and $sv\perp wu$,  the norm of the product on the left-hand side is $\sqrt{2}\sqrt{2}=2$, a contradiction.
\end{proof}

\cref{prop:no16} actually implies that no Euclidean Hurwitz algebra exists for $n>8$ (including $n=\infty$). Indeed, any extension of the octonions will break MoN. 

{
\section{Historical Notes}\label{sec:history}
In this  section we discuss the relationship of Gadi's development of $\C,\H$, and~$\O$ to two celebrated results by Frobenius (1878) and Hurwitz (1922)  {and their extensions}. 
 {More background on the history of these classical results   can be found in the survey by Baez~\cite{Baez02} 
and the book by Conway and Smith~\cite{conway2003quaternions}, and references within.}

\subsection{Hurwitz Theorem}
The results presented in this article are equivalent to a classical theorem by Hurwitz~\cite{Hurwitz1922} that classifies the Euclidean Hurwitz algebras.
	Recall that Euclidean Hurwitz algebras are real vector spaces with a multiplication that satisfies \Cref{axiom:dist}, \Cref{axiom:homo}, and \Cref{axiom:mon}.
 {A more common (but equivalent) definition is that of a real vector space with a bilinear multiplication that satisfies \Cref{axiom:mon}, and for which there exists a unit.}
	\footnote{Some authors define Hurwitz algebras without the requirement of a unit element, in which case \Cref{thm:hurwitz} should be stated with respect to unital Hurwitz algebras.}  
\begin{theorem}[Hurwitz (1922) \cite{Hurwitz1922}]\label{thm:hurwitz}
The only Euclidean Hurwitz algebras are $\R,\C,\H,$ and $\O$.
\end{theorem}

Thus, Gadi's development of $\C,\H$, and $\O$ provides an elementary proof of \Cref{thm:hurwitz}  {which is based on classical geometry.}
 {Previous proofs of this theorem exploit more advanced tools and require familiarity with modern algebra.}

\Cref{thm:hurwitz} was extensively generalized: Albert has shown that its conclusion continues to hold even if one replaces the Euclidean  {norm} in \Cref{axiom:mon} with an arbitrary norm~\cite{Albert47}, and Urbanik and Wright further extended  Albert's result to {\it infinite} dimensional algebras~\cite{Urbanik60}.

It is interesting to note that the requirement of a unit is essential:
	indeed, consider the following multiplication over $\R^2$, denoted by $\heartsuit$:
	\[ (a,b)\heartsuit (c,d) = (ad+bc,ac-bd).\]
	Thus, $\heartsuit$ is a simple modification of the standard complex multiplication
	where one  {swaps} the first and second coordinates of the result.
	One can verify that besides the existence of a unit, this multiplication satisfies all the requirements of \Cref{thm:hurwitz}  (in fact, it is even commutative!). 
	We refer the reader to \cite{Althoen83} for an accessible classification of related $2$-dimensional algebras.

\subsection{Frobenius Theorem}
Another classical result that is closely related to this article is due to Frobenius~\cite{F1878}.
	Here, a real division algebra is a (finite-dimensional) real vector space $V$ with a bilinear  {multiplication that} has no {\it zero-divisors}: 
	$(\forall a,b\in V): a\cdot b=0 \implies (a=0 \text{ or } b=0)$.
\begin{theorem}[Frobenius (1878) \cite{F1878}]\label{thm:frobenius}
	The only associative real division algebra are $\R,\C,\H$.
\end{theorem}
Thus, unlike \Cref{thm:hurwitz}, here one does not pose any ``geometric'' assumption such as \Cref{axiom:mon}.
	On the other hand, the multiplication is required to be associative: $(\forall a,b,c): (a\cdot b)\cdot c = a \cdot (b\cdot c)$.\footnote{\Cref{thm:frobenius} has another assumption which is absent in \Cref{thm:hurwitz}: $(\forall a,b\in V): a\cdot b=0 \implies (a=0 \text{ or } b=0)$. However, one can show that \Cref{axiom:mon} implies this assumption using the fact that $\| x\| = 0 \iff x=0$ for every $x\in\E^n$.}
	Furthermore, all the proofs of \Cref{thm:frobenius} we are aware of use algebraic tools that are beyond the scope of this manuscript. 
	A rather elementary proof of \Cref{thm:frobenius} (which still uses the fundamental theorem of algebra) 
	had been obtained by Richard Palais  (\cite{Palais1968}, see also Theorem 3.12 in \cite{Lam01}).
	It would be interesting to find elementary proofs at the level of Gadi's proof of \Cref{thm:hurwitz} that was presented in here.
	\Cref{thm:frobenius} was extended by Adams~\cite{adams1958} and by Bott and Milnor~\cite{bott1958}
	who showed that real division algebras may exist only in dimensions $1,2,4,8$.
	Their proofs heavily exploit topological arguments.
	We comment that there exist such algebras that are nonisomorphic to $\R,\C,\H$, or $\O$
	(the multiplication $\heartsuit$ defined above gives rise to such an algebra).}

\section{Conclusion}

In this article we have presented Gadi Moran's construction of complex numbers, quaternions, and octonions. The proofs used in this construction are based solely on properties of Euclidean geometry, and more specifically on the Pythagorean theorem. As such, this construction is more intuitive and is likely to be more accessible to high school students than standard constructions that derive Euclidean Hurwitz algebras from a purely algebraic standpoint.

Interestingly, our construction uses the properties of extended multiplication in $\E^n$; we need not ponder on their geometrical significance. As it is well known, multiplication by imaginary numbers can be visualized as consisting of rotations in $n$-dimensional space. For instance, multiplication of a complex number by $i$ performs a rotation by 90 degrees in the plane. This idea, when applied to quaternions, results in very useful techniques to model rotations in 3-dimensional space, which are most commonly used in computer graphics.

\section*{Acknowledgments}
We would like to extend our special thanks to Arik Moran for bringing Gadi's manuscript to the first author's attention and thus initiating this work, and also for his helpful remarks and continuous support during the writing of this paper.  {We would also like to thank Yuval Filmus and Ron Holzman for their insightful comments. Last but not least we would like to thank  {the anonymous referees} for many helpful comments and for drawing our attention to Richard Palais proof of Frobenius's theorem.}

\bibliographystyle{plain}
\bibliography{GadiPaperAMM_R2}

\begin{thebibliography}{10}

\bibitem{adams1958}
J.~F. Adams.
\newblock On the nonexistence of elements of hopf invariant one.
\newblock {\em Bull. Amer. Math. Soc.}, 64(5):279--282, 09 1958.

\bibitem{Albert47}
A.~A. Albert.
\newblock Absolute valued real algebras.
\newblock {\em Annals of Mathematics}, 48(2):495--501, 1947.

\bibitem{Althoen83}
Steven~C. Althoen and Lawrence~D. Kugler.
\newblock When is r2 a division algebra?
\newblock {\em The American Mathematical Monthly}, 90(9):625--635, 1983.

\bibitem{Baez02}
J.~C. Baez.
\newblock The octonions.
\newblock {\em Bull. Amer. Math. Soc.}, 39:145--205, 2002.

\bibitem{bott1958}
R.~Bott and J.~Milnor.
\newblock On the parallelizability of the spheres.
\newblock {\em Bull. Amer. Math. Soc.}, 64:87--89, 05 1958.

\bibitem{cayley1845}
Arthur Cayley.
\newblock {XXVIII. On Jacobi's Elliptic functions, in reply to the Rev. Brice
  Bronwin; and on Quaternions : To the editors of the Philosophical Magazine
  and Journal}, March 1845.

\bibitem{conway2003quaternions}
J.H. Conway and D.A. Smith.
\newblock {\em On Quaternions and Octonions}.
\newblock Ak Peters Series. Taylor \& Francis, 2003.

\bibitem{F1878}
Georg Frobenius.
\newblock {\"U}ber lineare substitutionen und bilineare formen.
\newblock {\em J. Reine Angew. Math}, 84:1--63, 1878.

\bibitem{graves1845}
John~T Graves.
\newblock Xlvi. on a connection between the general theory of normal couples
  and the theory of complete quadratic functions of two variables.
\newblock {\em The London, Edinburgh, and Dublin Philosophical Magazine and
  Journal of Science}, 26(173):315--320, 1845.

\bibitem{Hurwitz1922}
A.~Hurwitz.
\newblock {\"U}ber die komposition der quadratischen formen.
\newblock {\em Mathematische Annalen}, 88(1):1--25, 1922.

\bibitem{Lam01}
T~Y Lam.
\newblock {\em {A first course in noncommutative rings; 2nd ed.}}
\newblock Graduate texts in mathematics. Springer, New York, NY, 2001.

\bibitem{gadis}
Gadi Moran.
\newblock Real normed algebras revisited, December 2015.
\newblock Unpublished manuscript,
  www.dropbox.com/s/9c0vonev57ypciq/Gadi\%20Moran\%20Last\%20Paper.pdf?dl=0.
  Also available upon request.

\bibitem{Palais1968}
R.S. Palais.
\newblock The classification of real division algebras.
\newblock {\em American Mathematical Monthly}, 75:366--368, 1968.

\bibitem{Urbanik60}
K.~Urbanik and F.~B. Wright.
\newblock Absolute valued algebras.
\newblock {\em Proc. Amer. Math. Soc.}, 11:861--866, 1960.

\bibitem{vdw76}
B.L. van~der Waerden.
\newblock Hamilton's discovery of quaternions.
\newblock {\em Mathematics Magazine}, 49(5):227--234, 1976.

\end{thebibliography}
\end{document}